\documentclass[12pt,a4paper,onecolumn]{IEEEtran}

\usepackage{ multirow,array,ragged2e}
\usepackage{soul}
\usepackage{amsfonts}
\usepackage{subfig}
\usepackage{hyperref}

\usepackage{amsthm}
\usepackage{verbatim}
\usepackage{amsmath}
\usepackage{amscd}
\usepackage[latin2]{inputenc}
\usepackage{t1enc}
\usepackage[mathscr]{eucal}
\usepackage{indentfirst}
\usepackage{graphicx}
\usepackage{graphics}
\usepackage{pict2e}
\usepackage{epic}

\numberwithin{equation}{section}
\usepackage[margin=2.9cm]{geometry}
\usepackage{epstopdf} 
\usepackage{cite}
\usepackage{color}
\usepackage{marginnote}
\usepackage{MnSymbol,wasysym}
\usepackage{setspace}

\newtheorem{theorem}{Theorem}
\newtheorem{definition}[theorem]{Definition}
\newtheorem{assumption}[theorem]{Assumption}
\newtheorem{remark}[theorem]{Remark}
\newtheorem{lemma}[theorem]{Lemma}
\newtheorem{example}[theorem]{Example}
\newtheorem{corollary}[theorem]{Corollary}

\newcommand {\R}{} \def\R{\ensuremath{\mathbb{R}}}
\newcommand {\N}{} \def\N{\ensuremath{\mathbb{N}}}
\newcommand {\Z}{} \def\Z{\ensuremath{\mathbb{Z}}}

\def\FF{\mathcal{F}}
\def\BB{\mathcal{B}}

\def\P{\mathbb{P}}
\def\dist{\mathop{\rm dist}}

\definecolor{darkgreen}{rgb}{0,0.6,0}

\newcommand{\be}{\begin{equation}}
\newcommand{\ee}{\end{equation}}

\graphicspath{{../figures/}}

\begin{document}

\title{Random attraction in TASEP with time-varying hopping  rates\thanks{ This research is partially supported by grants from the DFG, project number 470999742, and the ISF.}}

\author{Lars Gr\"une, Kilian Pioch, Thomas Kriecherbauer, and Michael Margaliot\thanks{LG, KP, and TK are  with the
Mathematical Institute, University of Bayreuth, Germany.
MM is with the  School of Electrical \& Computer Engineering and the Sagol School of Neuroscience, Tel Aviv
University, Israel. }}

\maketitle

\onehalfspace 

\begin{abstract} 
 The totally asymmetric simple exclusion principle~(TASEP) is a fundamental model in nonequilibrium statistical mechanics. It describes the stochastic  unidirectional  movement of particles along a 1D chain of ordered sites. 
We consider the continuous-time version of TASEP with a finite number of sites  and  
   with time-varying hopping rates between the sites. We show how to formulate this model as a 
   nonautonomous random dynamical system~(NRDS) with a finite state-space. 
 We provide conditions guaranteeing that 
 random pullback and forward attractors of  such an NRDS exist and consist of singletons. In the context of the nonautonomous TASEP these conditions imply   almost sure synchronization of the individual random paths. 
 This implies in particular that perturbations that change the state of the particles along the chain
are  ``filtered out'' in the long run.
 We demonstrate that the required conditions are tight by providing examples where these conditions do not hold and consequently the forward  attractor does not exist or the pullback attractor is not a singleton. The results in this paper generalize our earlier results for autonomous TASEP in \cite{GrKM21} and contain these as a special case.
\end{abstract}

\maketitle


\section{Introduction} 

The totally asymmetric simple exclusion principle~(TASEP)
is  a fundamental model in nonequilibrium statistical mechanics~\cite{KrieKrug2010,Chou_2011}. It includes a 1D chain of ordered  sites. Each site can be either empty or occupied by a particle. In the totally asymmetric  model,
the particles  hop stochastically from left to right, but a particle cannot hop into an occupied site (simple exclusion). This generates an indirect coupling between the particles,  and allows to model the evolution of traffic jams  along the chain. Indeed, if a particle is ``stuck'' in a site for a long time then the particles behind it cannot move forward and thus start to accumulate in the preceding sites.  
Derrida et al.~\cite{Derrida_etal_1993} derived a matrix-form    
expression for the steady state under some assumptions (e.g., when  all the internal hopping rates are time-invariant and equal), and used it to analyze phase transitions in TASEP.

 TASEP was originally introduced to model biological processes such as mRNA translation by ribosomes~\cite{MacDonald1968}.
For a survey on using TASEP to model and analyze protein synthesis, see~\cite{TASEP_tutorial_2011}.
However, TASEP and its variants have found numerous applications in diverse fields e.g. modeling vehicular traffic flow and pedestrian dynamics~\cite{ScCN11}, 
the movement of protein motors along filaments in the cell~\cite{TASEP_motors,Muller_2005}, wireless line networks~\cite{TASEP_COMM_NETS},
and more. In all these applications, traffic jams of particles play an important role. For example, 
Leduc et al.~\cite{Leduc_motor_protein_traffic_jams} demonstrate the generation of 
density- and bottleneck-induced traffic jams
 of motor proteins, and it is known that  
 defective intracellular transport can  trigger neuron degeneration~\cite{axonal_transport}.

The random times at which a particle attempts to hop are determined by Poisson processes attached to the sites of the chain,  and the dynamic behavior of TASEP depends on the transition rates of these processes into, out of, and along the chain. 
There are good reasons to study  TASEP with time-varying hopping rates. For example, 
traffic flow is often regulated by   signals  that vary with time (e.g., traffic lights). 
  In a model for ribosome flow along the mRNA molecule time-varying rates are relevant  because
the  eukaryotic cell cycle is controlled by periodic gene expression~\cite{peri_cell_cycle,frenkel12,Higareda2010}.

For TASEP with jointly~$T$-periodic   hopping rates, the results from~\cite{entrain_master} imply that the probability distribution of the solutions converges to a unique $T$-periodic distribution, which is independent of the initial distribution. In other words, the distribution entrains to a periodic solution and distributions with different initial conditions synchronize asymptotically.
The dynamic mean-field of TASEP, called the ribosome flow model~(RFM)~\cite{RFM_stability}, is a deterministic ODE model, and it is known that if the transition rates in this model are jointly $T$-periodic then the densities at the sites converge to a unique~$T$-periodic pattern independent of the initial condition~\cite{RFM_entrain}.

 This immediately raises the question whether the individual paths of the stochastic TASEP model also entrain or synchronize in a suitable sense. In this paper, we give a positive answer to this question, using the framework of nonautonomous random dynamical systems (NRDS) and nonautonomous  random
 attractors. The results and the conditions on the system are formulated for general finite-state NRDS,  and are specialized to TASEP with both general time-varying and with periodic hopping rates.

 The main contributions of this paper include:
\begin{enumerate}
    \item We show how to formulate TASEP with time-varying hopping rates as an NRDS.
    \item We provide conditions under which nonautonomous random pullback and forward attractors exist for general finite-state NRDS.
    \item We provide conditions under which these nonautonomous random attractors consist  of singletons, implying almost sure synchronization of the paths. 
    \item We specialize our results for TASEP with general time-varying and periodic hopping rates. In particular, guaranteeing that attractors   are singletons   implies an important robustness property for the nonautonomous TASEP:   perturbations that change the state of the particles along the chain
    are ``filtered out'' in the long run. 

    \item  We demonstrate  that our conditions are tight by providing   examples  where the
    conditions are violated, and  consequently   forward attraction   fails to hold or the pullback attractor does  not consist of singletons.
\end{enumerate}

Item 5) shows a significant difference to the autonomous case studied in \cite{GrKM21}, while Corollary~\ref{cor:periodic}  below
shows that the nonautonomous periodic case behaves qualitatively similar to the time-invariant case.
In particular, this shows that we obtain the results for autonomous TASEP from \cite{GrKM21} as a special case.

The remainder of this paper is organized as follows. The next section briefly reviews~NRDSs. 
Section~\ref{sec:TASEP}  shows how continuous-time TASEP with a finite number of sites can be formulated  as an NRDS with a finite state space.  
Section~\ref{sec:attractors}  presents results on    nonautonomous random attractors of  a general~NRDS with a  finite state space. The subsequent section applies these results to the  particular case of~TASEP with time-varying rates. Section~\ref{sec:NumSim} describes numerical simulations that demonstrate the theoretical results. The final section concludes.

\section{Nonautonomous random dynamical systems}

In order to   analyze  random attraction in TASEP with time-varying jump rates, we make use of the framework of nonautonomous random dynamical 
systems~(NRDS). The definition of autonomous random dynamical systems~(RDS) goes back at least as far as~\cite{ArnC91}, and a comprehensive treatment can be found in the monograph \cite{Arno98}.  Here we use the  nonautonomous extension from \cite[Definition 1]{CuiK18}. For its definition, we let 
$(X,d_X)$ be a Polish metric space, $(\Omega,\FF,\P)$ a probability space, and $\R^2_{\ge}:=\{(t,s)\in\R^2\,|\, t\ge s\}$. Let~$\BB(X)$ and~$\BB(\R^2_\ge)$   denote the Borel sigma algebra on $X$ and $\R_\ge^2$, respectively. 
\begin{definition}
A mapping $\phi:\R_\ge^2  \times X\times\Omega \to X$ is called a {\em nonautonomous random dynamical 
system}~(NRDS) if 
\begin{enumerate}
\item[(i)] $\phi$ is $(\BB(\R_\ge^2)\times\BB(X)\times\FF,\BB(X))$-measurable;
\item[(ii)] $\phi(t,t,\cdot,\omega)$ is the identity on $X$ for all $t\in\R$ and $\omega\in\Omega$;
\item[(iii)] $\phi$ satisfies the {\em cocycle property}
\[ \phi(t,t_0,x,\omega) = \phi(t,s,\phi(s,t_0,x,\omega),\omega) \; \mbox{ for all } 
 t\ge s\ge t_0,\; x\in X, \;\omega\in\Omega;\]
\item[(iv)] $
x \mapsto \phi(t,s,x,\omega)$ is continuous.
\end{enumerate}
\label{def:nrds}\end{definition}
The interpretation of this definition is as follows: $\phi(\cdot,t_0,x,\omega)$ denotes the solution path corresponding to $\omega$ starting at time $t_0$ in state $x$. 
We note that in contrast to the autonomous RDS definition we have $\omega$ and not $\theta_s\omega$ as an argument of the ``outer'' $\phi$ on the right hand side of the cocycle property (iii)  (for a definition of $\theta_s\omega$ for the autonomous TASEP model we refer to \cite{GrKM21}). 
This is because the time shift that is caused by $\theta_s$ is carried out in the maps $\phi$, which---in contrast to the autonomous case---depends explicitly on the initial time $t_0$. 
As a consequence, we do not need an auxiliary dynamical system~$\theta$ on the probability space.

\section{TASEP with time-varying rates as a nonautonomous random dynamical system}\label{sec:TASEP}

The \emph{totally asymmetric simple exclusion process}~(TASEP)
 is a Markov process for particles hopping or jumping along a 1D chain. We consider the continuous-time version of TASEP here. Moreover, we restrict ourselves to finite lattices with $n\in\N$ sites. Then the Markov process has only a finite number of states. A particle at site $k\in\{1,\ldots,n-1\}$ hops
to site $k+1$ (the next site on the right)
  at a random jump time\footnote{Actually, ``jump attempt times'' would be the more accurate name, but as it is also more clumsy so we prefer the shorter ``jump times''.\label{foot:actual}}  that is exponentially distributed with rate $\lambda_k$, provided that site $k+1$ is not occupied by another particle. 
	This simple exclusion property generates an indirect coupling between the particles and allows, e.g.,
	to model the formation of traffic jams. Indeed, 
	if a particle ``gets stuck'' for a long time in the same site then other particles accumulate behind it. 
At the left end of the chain  particles enter with the entry rate~$\lambda_0>0$, and at the right end particles leave site~$n$ with a rate~$\lambda_{n}>0$. We refer to \cite{Ligg99}, \cite{ScCN11} and the references therein for more information about this model.

Here we consider the variant of TASEP in which the jump rates~$\lambda_k$ assigned to the sites $k\in\{1,\ldots,n-1\}$ as well as the entry rate $\lambda_0$ and exit rate $\lambda_n$ vary with time. To this end, we modify the definition of the~RDS corresponding to the time-invariant TASEP in~\cite{GrKM21} to an~NRDS. This requires  shifting from   homogeneous Poisson processes  to  non-homogeneous ones. 

\subsection{ Homogeneous and non-homogeneous Poisson processes}\label{sec:nonhompoisson}

First observe that for each of the sites as well as for the entry and exit the sequence of random jump times is given by a Poisson process, which for our case of time-varying rates is a non-homogeneous Poisson process.
The initial times $t_0$ of RDS and NRDS are arbitrary real numbers, which in particular includes arbitrary negative real numbers. This is an important feature for the construction of the random attractors later on in the paper. As a consequence, we need a Poisson process on the whole real line $\R$. This is defined in the following two definitions, which follow  \cite[Sections 1.3, 2.1, and 4.5]{Kingman}.

\begin{definition}\label{def:Ppp}
Let $(\Omega,\mathcal F,\P)$ be a probability space, $|\cdot |$ the Lebesgue measure 
on~$\R$, $\BB(\R)$ the set of Borel sets $A\subset\R$, $\lambda>0$, and let~$\R^{\infty}$ denote the set of all countable subsets of~$\R$.
 A homogeneous  Poisson process on~$\R$ with rate~$\lambda$
 is a map~$\Pi : \Omega \to \R^{\infty}$ satisfying the following 
 three conditions:
\begin{itemize}
\item[(i)] The maps $N(A):\Omega \to \N \cup \{ \infty \}$, $\omega \mapsto \#  (\Pi(\omega) \cap A)$ are measurable for all $A \in \BB(\R)$, i.e.~for all $m \in \N \cup \{ \infty \}$ and all $A \in \BB(\R)$ we have that the set $\{ \omega \in \Omega \mid \Pi(\omega) \cap A$ contains exactly~$m$ points$\}$ 
belongs to the sigma algebra $\mathcal F$. 
\item[(ii)] For any pairwise disjoint sets $A_1,\ldots,A_p\in\BB(\R)$, $p\in\N$, the random variables $N(A_1),\ldots,N(A_p)$ are independent.
\item[(iii)] $N(A)$ is ${\rm Pois}(\lambda |A|)$-distributed for all $A\in\BB(\R)$, that is,  $\P(N(A)=k) = \frac{(\lambda|A|)^k}{k!}e^{-\lambda|A|}$ for each integer $k\ge 0$. 
\end{itemize}
\end{definition}

In the following definition we use the notation~$\R_0^+ := \{x\in\R\,|\, x \ge 0\}$.
\begin{definition}\label{def:inhomPpp}
Let $(\Omega,\mathcal F,\P)$ be a probability space.
 A non-homogeneous Poisson process on~$\R$ with  locally integrable non-negative rate  function~$\lambda:\R\to\R_0^+$
 is a map~$\Pi : \Omega \to \R^{\infty}$, such that for the function 
 \[ M(t):=\int_0^t \lambda(s)ds, \]
 the map $M^{\infty}\circ \Pi$ is a homogeneous Poisson process with rate $1$ on~$\R$,  where $M^{\infty}$ means that $M$ is applied to each component of $\Pi(\omega)$.
\end{definition}
As $\lambda(t)=0$ is possible, $M$ may not be invertible on the whole $\R$, i.e., there may be points $y$ for which $M^{-1}(\{y\})$ is an interval. 
However, according to \cite[Section 4.5]{Kingman}, this can only happen for at most countably many $y$, hence the probability of $\Pi$ attaining such a point is zero,   and we may remove the corresponding event from the set $\Omega$. Hence, one can use $M^{-1}$ for passing to a non-homogeneous Poisson process from a homogeneous one.
Note that Definition \ref{def:inhomPpp} implies that $\lim_{t\to\infty} M(t) = \infty$, $\lim_{t\to-\infty} M(t)=-\infty$, because otherwise $ M^{\infty} \circ \Pi$ cannot be a Poisson process on $\R$.

Intuitively, the larger $\lambda(t)$ is in a time interval $I$, the larger $N(I)$ is expected to be in this interval. More precisely, the probability that exactly $k$ events occur on $I=[t_1,t_2]$ is
\[ \P(N(I)=k) = \frac{\Lambda^k}{k!}e^{-\Lambda},  \text{ with } \Lambda = \int_{t_1}^{t_2}\lambda(t)dt.\]

For each $\omega\in\Omega$ the map $\Pi$ then defines a sequence of jump times~$T_k(\omega)$, $k\in\Z$. In fact, as shown in \cite[Definition 5]{GrKM21}, we can explicitly construct a probability space~$(\Omega,\FF,\P)$ for the Poisson process by identifying $\omega$ with the sequence of times $\xi_i:=\hat T_i(\omega)-\hat T_{i-1}(\omega)$, with~$\hat T_i$ collecting all positive jump times for~$i > 0$, all negative times for~$i<0$, and~$\hat T_0:=0$. The construction in \cite[Definition 5]{GrKM21} is made for the homogeneous Poisson process but easily carries over to the non-homogeneous case. By construction, it excludes accumulation points of the sequence~$(T_i(\omega))_{i\in\Z}$ and identical jump times~$T_i(\omega)=T_{i-1}(\omega)$ for each~$\omega\in\Omega$.

For TASEP, we need one non-homogeneous Poisson process for each site, including the entry and exit, whose rates we denote by $\lambda_0,\ldots,\lambda_{n}$, where the indices $0$ and $n$ stand for entry to and exit from the chain, respectively. To this end, we take the $(n+1)$-fold product of the probability space for a single process. Any $\omega$ then corresponds to $n+1$ stochastically independent point processes, which we may represent by strictly increasing sequences $(T_{k,j}(\omega))_{j\in \Z}$ that are unbounded above and below. Here $k=0,\ldots,n$ denotes the lattice site of the random clock where $k=0$ represents the clock for particles entering the first site. Since the exponential distribution is absolutely continuous it is not hard to see that the event that there exist $k\neq k'$, $j$, $j'$ with $T_{k,j}(\omega)=T_{k',j'}(\omega)$ has zero probability,  and we remove this event from our probability space. 

By construction, all jump times $T_{k,j}(\omega)$ are then pairwise distinct for all $\omega \in \Omega$. Therefore there exist unique sequences $k_i=k_i(\omega)$ and $j_i=j_i(\omega)$, $i \in \Z$, with 
\[ T_{k_i,j_i}(\omega) < T_{k_{i+1},j_{i+1}}(\omega)\,, i \in \Z\,, \quad \text{and} \quad T_{k_{-1},j_{-1}} < 0 \;\leq T_{k_{0},j_{0}} \,.\]
We call the random sequence $(k_i)_i$ the {\em jump order sequence} with corresponding {\em jump time sequence} $t_i := T_{k_i,j_i}$.

We remind again that   the jumps at these times indeed take place only if a particle hops to an empty site (see Footnote~\ref{foot:actual}).

\subsection*{Definition of the TASEP  NRDS}

Our goal now is to define an NRDS (that is, a  mapping~$\phi$ that satisfies all the conditions in Definition~\ref{def:nrds})  that corresponds to TASEP with time-varying rates.
 Our definition follows the steps in \cite{GrKM21}, which are suitably adapted to the nonautonomous setting of this paper.

For the convenience of the reader, we summarize this construction in a shortened but self-contained way.

In order to define the state of the system we 
associate to each site $k\in\{1,\ldots,n\}$ a
variable~$s_k$. 
We set $s_{k}=1$ if site $k$ is occupied by a particle,  and $s_k=0$ if it is not. Hence, the (finite) state space of TASEP is $X=\{0,1\}^n$. Since the state space is finite, we use the discrete topology and its Borel sigma algebra, i.e.~all subsets of $X$ are open and measurable.

For formalizing a single particle hop in TASEP, we define the following map $f$. We are given a state $x = (s_1,\ldots,s_n)\in X$ and an index $k\in\{0,\ldots,n\}$ of the site at which the particle attempts to hop, where $k=0$ represents a particle entering the chain. Then we define 
\[ f(x,k) := \left\{ \begin{array}{ll} 
(1,s_2,\ldots,s_n) & \mbox{ if } k=0 \\
(s_1,\ldots,s_{n-1},0) & \mbox{ if } k=n \\
(s_1,\ldots,s_{k-1},0,1,s_{k+2},\ldots, s_{n}) & \mbox{ if } k\ne 0, k\ne n, s_k=1 \mbox{ and } s_{k+1}=0\\
x & \text{ otherwise.} 
\end{array}\right.\]

Note that the first line in the definition of~$f(x,k)$ corresponds to the case where a particle attempts to hop into the first site. Then two cases are possible: if the first site is empty [full] then the  particle indeed hops [does not hop].  Yet, in both cases we end up with the configuration~$(1,s_2,\dots,s_n)$.

Now assume that we have a sequence of jump times $(t_i(\omega))_{i\in\Z}$ with $t_i(\omega)\in\R$ and $t_{i}(\omega) < t_{i+1}(\omega)$ for all $i\in\Z$ together with its jump order sequence $k_i(\omega)\in \{0,\ldots,n\}$. The transition $\tilde\phi:\R_\ge^2\times X\times\R^\Z\times\{0,\ldots,n\}^\Z\to~X$ mapping the initial value $x_0$ at initial time $t_0$ to the state $\tilde\phi(t,t_0,x_0,(t_i(\omega)),(k_i(\omega)))$ at time $t$, given the jump time and order sequences $(t_i(\omega))_{i\in\Z}$ and $(k_i(\omega))_{i\in\Z}$ is then defined by 
\[
\tilde\phi(t,t_0,x_0,(t_i(\omega)),(k_i(\omega))):=x_0
\]
if $t_i(\omega) \not\in [t_0, t)$ for all $i \in \Z$, otherwise inductively via 
\begin{equation} x_{p+1} := f(x_p,k_{p+i_0}) \mbox{ for } p=0,\ldots,\Delta i, \qquad \tilde\phi(t,t_0,x_0,(t_i(\omega)),(k_i(\omega))):= x_{\Delta i+1} \label{eq:tildephi}\end{equation}
where $i_0:=\inf\{ i\in\Z\,|\, t_i(\omega) \ge t_0\}$, $i_1:=\sup\{ i\in\Z\,|\, t_i(\omega) < t\}$ and $\Delta i: =i_1-i_0$. Here we use that $(t_i(\omega))$ has no accumulation points in $\R$.

Finally, we can define the mapping $\phi$ using $\tilde\phi$ of \eqref{eq:tildephi} and the just defined sequences of jump order $(k_i)_i$ and jump times  $(t_i)_i$:
\[ \phi(t,t_0,x,\omega) := \tilde\phi(t,t_0,x,(t_i(\omega)),(k_i(\omega)))\,. \]
Let us check the requirements of Definition \ref{def:nrds}. There is nothing to show for condition (iv), because the state space is discrete. Condition (i) follows from the construction and (ii) holds because $\tilde\phi$ leaves $x$ unchanged for $t=0$ (there is no jump time $t_i(\omega)$ in $[0, t) = \emptyset$). The cocycle property (iii) is an immediate consequence of the inductive definition of $\tilde\phi$.

\section{Random attractors for finite state nonautonomous random dynamical systems}\label{sec:attractors}

As we have seen, TASEP with time-varying rates can be formulated 
 as an NRDS with a finite state space. In this section, we present results for attraction and in particular nonautonomous random attractors of general NRDS with finite state space. While these results are interesting in their own right, in the subsequent section we will in particular use them for TASEP.

\subsection{Nonautonomous random attractors}

Nonautonomous random attractors describe the asymptotic  behavior of systems subject to a stochastic and  time-varying influence. 
The attractors can provide insight into the stability and persistence of certain     behaviors of
the system. We use the following definitions of random attractors in the pullback and in the forward sense. We refer to \cite{ChKS02,Sche02} for a study of the difference between pullback and forward attraction. Here we limit ourselves to the definition of {\em global} random attractors. The first step is to recall  the notion of a  compact nonautonomous random set.

\begin{definition} A {\em nonautonomous random set} $C$ on a probability space $(\Omega,\FF,\P)$ is a family $(C(t))_{t\in\R}$ of time-dependent measurable subsets of $X\times\Omega$ with respect to the product $\sigma$-algebra of the Borel $\sigma$-algebras of $X$ and $\FF$. The~$\omega$-section of a random set~$C$ at time $t\in \R$ is for each $\omega\in\Omega$ defined by
\[ C(t,\omega) = \left\{ x\in X\,|\, (x,\omega)\in C(t)\right\}.\]
The nonautonomous random set is called compact if every $C(t,\omega)$ is compact.
\end{definition}

\begin{definition} Let $\phi$ be an NRDS on $(\Omega,\FF,\P)$. A 
compact nonautonomous random set~$A$, which is strictly $\phi$-invariant, i.e., 
\[ \phi(t,t_0,A(t_0,\omega),\omega) = A( t,\omega) \; \mbox{ for all } (t,t_0)\in\R_\ge^2 \mbox{ a.s.}, \]
is called a {\em global random pullback attractor}, if for each $t\in\R$
\[ \lim_{t_0\to-\infty} \dist\big(\phi(t, t_0, X, \omega), A(t,\omega)\big) = 0 \mbox{ a.s.}. \]
It is called a \emph{global random forward attractor} if for each $t_0\in\R$
\[ \lim_{t\to\infty} \dist\big(\phi(t, t_0, X, \omega), A( t,\omega)\big) = 0 \mbox{ a.s.}. \]
 Here $\dist(A_1,A_2) := \sup_{a_1\in A_1}\inf_{a_2\in A_2}d(a_1,a_2)$. 
\end{definition}

In our case,  the finite state space $X$ is equipped with the discrete topology and we may therefore use the distance defined by $d(x_1,x_2)=1$ if~$x_1\ne x_2$ and~$d(x_1,x_2)=0$ if~$x_1= x_2$. This implies for subsets $A$, $B \subset X$ that $d(A,B)=0$ if $A\subset B$ and $d(A,B)=1$, otherwise. 

For the construction of the attractor, we use that for all $t_1 < t_0 < t$ the cocycle property implies
\begin{equation} \phi(t,t_1,X,\omega) = \phi(t,t_0,\phi(t_0,t_1,X,\omega),\omega) \subset \phi(t,t_0,X,\omega), \label{eq:cocyclsts}\end{equation}
so the set $\phi(t,t_0,X,\omega)$ is decreasing w.r.t.\ set inclusion for decreasing $t_0$. Hence, for each $t\in\R$ we can define its set valued limit via
\begin{equation} A(t,\omega) := \bigcap_{ t_0\le t} \phi(t,t_0,X,\omega). \label{eq:Adef}\end{equation}

\begin{theorem} \label{thm:rattr}
Consider an NRDS with  a  finite state space. Then $A(t,\omega)$ from \eqref{eq:Adef} is nonempty for all $\omega\in\Omega$, $t\in\R$,  and defines a global random pullback attractor. Moreover, for all $\omega\in\Omega$ and $t\in\R$ there exists $T_0(t,\omega)< t$ such that $\phi(t,t_0,X, \omega) = A(t,\omega)$ for all~$t_0\le T_0(t,\omega)$. If, in addition, for each $p\in(0,1)$ there is $T_p>0$ such that 
\begin{equation} \P(\{\omega\in\Omega\,|\, t-T_0(t,\omega)\le T_p\})\ge p \mbox{ holds for all $t\in\R$}, \label{eq:Tp}\end{equation} 
then $A(t,\omega)$ also defines a global random forward attractor and for almost every $\omega\in\Omega$ and all $t_0\in\R$ there exists $T(t_0,\omega)>t_0$ such that 
\[
\phi(t,t_0,X,\omega) = A(t, \omega) \text{ for all } 
t\ge T(t_0,\omega).
\]
\end{theorem}
\begin{proof}
For $t\ge t_0$ and $\omega\in\Omega$ define the set $B(t,t_0,\omega): = \phi(t,t_0,X,\omega)$. From \eqref{eq:cocyclsts} we obtain that for all $t_1<t_0<t$, $x\in X$ and $\omega\in\Omega$
\begin{equation} B(t,t_1,\omega) =  \phi(t,t_0,B(t_0,t_1,\omega),\omega) \subset B(t,t_0,\omega). \label{eq:Bts}\end{equation}
Since these sets are finite, this inclusion implies  for all $t \in \R$ and $\omega\in\Omega$ that the map~$t_0\mapsto 
B(t,t_0,\omega)
$ can change its value only finitely many times, implying that $A(t,\omega) = \bigcap_{ t_0\le t} B(t,t_0,\omega)$ equals $B(t,t_0,\omega)$ for all sufficiently small $ t_0\le t$ and is thus in particular nonempty. 
 Moreover, there exists $T_0(t,\omega)<t$ such that
$A(t,\omega) = B(t,t_0,\omega)$ holds for all $t_0\le T_0(t,\omega)$. Then
\[\dist(\phi(t,t_0,X,\omega), A(t,\omega)) = \dist(B(t,t_0,\omega), A(t,\omega)) = 0 \]
for all $t_0\le T_0(t, \omega)$, i.e., finite time pullback attraction.

Next we prove $\phi$-invariance of $A$, i.e., $\phi(t,t_0,A(t_0,\omega),\omega) = A(t,\omega)$ for all $t\ge t_0$ and all $\omega\in\Omega$. To this end, fix $t\ge t_0$ and $\omega\in\Omega$ and choose $t_1\le \min\{ T_0(t,\omega), T_0(t_0,\omega)
\}$. Then we have $A(t,\omega)=B(t,t_1,\omega)$ and $A(t_0,\omega) = B(t_0,t_1,\omega)$. Using the first identity in \eqref{eq:Bts} this yields
\[ \phi(t,t_0,A(t_0,\omega),\omega) = \phi(t,t_0,B(t_0,t_1,\omega),\omega) = B(t,t_1,\omega)=A(t, \omega).\]
Together, this shows that $A$ is a global random pullback attractor and that pullback attraction happens in finite time $ t-T_0(t,\omega)$ for each $\omega\in\Omega$ and $t\in\R$.

In order to see that $A$ is also a global forward attractor under the additional condition~\eqref{eq:Tp}, fix $t_0\in\R$, $T>0$,  and consider the set 
\[ \Omega_{t_0,T} :=  \{ \omega\in\Omega \,|\, t_0 \le T_0(t_0+T,\omega)\}. \]
The definition of $T_0$ implies  for every $\omega\in \Omega_{t_0,T}$ that
\[ \phi(t_0+T,t_0,X,\omega) = A(t_0+T,\omega). \]
Using the cocycle property and the invariance of $A$ yields
\begin{equation} \phi(t,t_0,X,\omega) = A(t,\omega) \label{eq:AS}
\end{equation}
for all $t\ge t_0+T$ 
and all $\omega\in \Omega_{t_0,T}$. From \eqref{eq:Tp} it moreover follows that $\P(\Omega_{t_0,T_p})\ge p$ holds, implying that \eqref{eq:AS} holds with probability $\ge p$ for all  $t\ge t_0+T_p$.
The identity~\eqref{eq:AS} implies 
\[ \dist(\phi(t,t_0,X,\omega), A(t,\omega)) = 0 \]
for all $t\ge t_0+T_p$, and thus forward attraction in finite time $T_p$ with probability larger than $p$. Since $p\in(0,1)$ is arbitrary, this implies forward attraction to $A$ with arbitrarily large probability and thus almost sure forward attraction in finite time.
\end{proof}

\begin{remark} (i)
In contrast to the autonomous case studied in \cite{GrKM21}, here the finite state property does not imply that the pullback attractor defined 
in~\eqref{eq:Adef} is also
a forward attractor.  We display such a case in Example~\ref{weirdTASEP} in the context of the non-autonomous TASEP model.
Condition \eqref{eq:Tp} is a uniformity condition that safeguards against such a situation. In the next section we will derive sufficient conditions on the rates of TASEP implying \eqref{eq:Tp}.

(ii) Even under the uniformity condition \eqref{eq:Tp} there is an asymmetry between the statements for forward and pullback attraction in Theorem \ref{thm:rattr}: while pullback attraction holds for all $\omega\in\Omega$, forward attraction holds only for almost all $\omega\in\Omega$. As \cite[Remark 14(ii)]{GrKM21} shows, this is not a shortcoming of our proof but can happen even in the autonomous case.
\label{rem:rw3}
\end{remark}

\subsection{Random attractors consisting of single trajectories}

We now investigate conditions under which the random attractor consists of a single trajectory almost surely. In this case, the 
long time behavior of the solutions of the~NRDS is almost surely independent of the initial condition. The following theorem gives necessary and sufficient conditions for this property to hold.
For any~$t_0\leq t$, define the sets 
\begin{equation} \Gamma(t,t_0):= \{\omega\in\Omega\,|\,\phi(t,t_0,x_1,\omega)=\phi(t,t_0,x_2,\omega) \mbox{ for all } x_1,x_2\in X\}. 
\label{eq:gammadef}\end{equation}
Intuitively speaking, this is the set of random realizations for which  all solutions starting at time $t_0$  synchronize at time $t$,  regardless of their initial condition.
We first prove the following lemma.

\begin{lemma}
    For all $\hat t \ge t \ge t_0 \ge \hat t_0$ it holds that $\Gamma(t,t_0)\subset \Gamma(\hat t,\hat t_0)$.
\label{lemma:Gammamonotone}\end{lemma}
\begin{proof}
    Pick $\omega\in\Gamma(t,t_0)$ and $x_1,x_2\in X$. Define $\tilde x_i = \phi(t_0,\hat t_0,x_i,\omega)$ for $i=1,2$. Then the definition of $\Gamma$ yields
    \[ \phi(t,t_0,\tilde x_1,\omega) =  \phi(t,t_0,\tilde x_2,\omega). \]
    By the cocycle property we obtain
    \begin{align*}
        \phi(t,\hat t_0,x_1,\omega) = \phi(t, t_0,\tilde x_1,\omega) = \phi(t, t_0,\tilde x_2,\omega) = \phi(t,\hat t_0,x_2,\omega)
    \end{align*}
    and 
    \begin{align*}
        \phi(\hat t,\hat t_0,x_1,\omega) = \phi(\hat t, t ,\phi(t,\hat t_0, x_1,\omega),\omega) = \phi(\hat t, t ,\phi(t,\hat t_0, x_2,\omega),\omega) = \phi(\hat t,\hat t_0,x_2,\omega).
    \end{align*}
    This shows that $\omega \in \Gamma(\hat t,\hat t_0)$ and completes the proof.
\end{proof}

\begin{theorem} Consider an NRDS with a finite state space. Then statements (2)-(4) below are all equivalent to each other and imply statement (1).
\begin{enumerate} 
\item[(1)] There exists $p>0$ such that for any $t\in\R$ there exists $t_0<t$ with 
\[ \P(\Gamma(t,t_0)) > p. \]
\item[(2)] For any $t\in\R$ it holds that
\[ \lim_{t_0\to-\infty}\P(\Gamma(t,t_0)) = 1. \]
\item[(3)] For any $t\in\R$ 
\[ \P(\{\omega\in\Omega\,|\, \phi(t,t_0,X,\omega) \mbox{ is a singleton for some } t_0 < t\}) = 1.\]
\item[(4)]  For all $t\in\R$ the sets $A(t,\omega)$ from \eqref{eq:Adef} are singletons for almost all $\omega\in\Omega$.
\end{enumerate}
If in addition the solutions $\phi(t,t_0,x,\cdot)$ are stochastically independent on non-overlapping intervals, i.e.,~$\phi(t,t_0,x,\cdot)$ is independent of $\phi(s,s_0,x,\cdot)$ if $[t_0,t) \cap [s_0,s) = \emptyset$, then statement~(1) is equivalent to statements~(2)-(4), If, moreover, the $t_0$ in (1) can be chosen such that the difference $t-t_0$ in (1) is bounded uniformly in $t$, then the convergence in statement~(2) has an exponential rate, and the singletons $A(t,\omega)$ also provide a global random forward attractor. 
\label{thm:singleton}
\end{theorem}
\begin{proof}
We first show for a general finite state NRDS the implication (2) $\Rightarrow$ (1) as well as the equivalences (2) $\Leftrightarrow$ (3), and (3) $\Leftrightarrow$ (4). The proof is then completed by demonstrating the implication (1) $\Rightarrow$ (2)  under the additional assumption of stochastic independence of the RDS with an exponential rate of convergence under the stated uniformity condition on $t-t_0$.

(2) $\Rightarrow$ (1): This is obvious for any $p\in(0,1)$.

(2) $\Leftrightarrow$ (3): Since $\phi(t,t_0,X,\omega)$ is a singleton if and only if $\phi(t,t_0,x_1,\omega)=\phi(t,t_0,x_2,\omega)$  holds for all $x_1,x_2\in X$, we have 
\begin{equation}\label{eq:2to3}  
\bigcup_{t_0 < t} \, \Gamma (t, t_0) \; = \; \{\omega\in\Omega\,|\, \phi(t,t_0,X,\omega) \mbox{ is a singleton for some } t_0 <t\} \,.
\end{equation}
By the monotonicity of the sets $\Gamma(t,t_0)$ proved in Lemma \ref{lemma:Gammamonotone} it follows that 
\begin{equation*}
\P\left(\bigcup_{t_0 < t} \, \Gamma (t, t_0)\right) = \lim_{t_0\to-\infty}\P(\Gamma(t,t_0))
\end{equation*} 
 for any $t\in\R$
 and thus
\[ \lim_{t_0\to-\infty}\P(\Gamma(t,t_0))  \; = \; \P(\{\omega\in\Omega\,|\, \phi(t,t_0,X,\omega) \mbox{ is a singleton for some } t_0 <t\}),\]
which implies the equivalence between (3) and (2).

(3) $\Leftrightarrow$ (4): The definition of the set $A(t,\omega)$ in~\eqref{eq:Adef} together the finiteness of the state space gives (cf.~the proof of~Theorem \ref{thm:rattr}) 
\begin{equation}\label{eq:4to5}
 \bigcup_{t_0 < t} \, \Gamma (t,t_0)  \; = \;  \{\omega\in\Omega\,|\,A(t,\omega) \mbox{ is a singleton}\}\,.
\end{equation}
Hence, the sets on the right-hand-sides of relations \eqref{eq:2to3} and \eqref{eq:4to5} have the same probability. So far we showed that~$(4) \Leftrightarrow(3) \Leftrightarrow(2) \Rightarrow(1)$.

From now on we also assume the stochastic independence of the NRDS on non-overlapping time intervals.

(1) $\Rightarrow$ (2): From Lemma \ref{lemma:Gammamonotone} it follows  for all $t_1<t_2<t_3$ that 
\begin{equation} \Gamma(t_3,t_1) \supset \Gamma(t_3,t_2) \cup \Gamma(t_2,t_1). \label{eq:gammasupset}\end{equation}
Fixing $t\in\R$ for which we want to prove (2), because of the monotonicity of $t_0\mapsto \P(\Gamma(t,t_0))$ it suffices to prove the convergence in (2) for a suitable sequence $t_m\to-\infty$. 

To this end, observe that for the complements $A^C := \Omega \setminus A$, relation \eqref{eq:gammasupset} implies 
\begin{equation} \Gamma(t_3,t_1)^C \subset \Gamma(t_3,t_2)^C \cap \Gamma(t_2,t_1)^C \label{eq:gammasubset}\end{equation}
for all~$t_1<t_2<t_3$. Now we define a sequence $(t_m)_{m\ge 0}$ with $t_m \to -\infty$ as follows: We set~$t_0:=t$. Then, inductively for $m=0,1,2,\ldots$, given $t_m$ we pick $t_{m+1}<t_m$ such that (1) is satisfied for $t=t_m$ and $t_0=t_{m+1}$. By monotonicity of $t_{m+1}\mapsto \P(\Gamma(t_m,t_{m+1}))$ we can choose $t_{m+1}\le t_m-1$, thus ensuring $t_m\to -\infty$ as $m\to\infty$.
This choice of the $t_m$ implies
\[ \P(\Gamma(t_m,t_{m+1})^C) \le 1-p <1 \]
for all $m\ge 0$ with $p>0$ from (1). By the  additional
assumption on the NRDS, the 
sets~$\Gamma(t_m,t_{m+1})$ and thus the sets~$\Gamma(t_m,t_{m+1})^C$ are stochastically independent for different~$m$. Thus, using \eqref{eq:gammasubset} we obtain
\[ \P(\Gamma(t,t_{m})^C) \le \P\left(\bigcap_{\ell=1}^m\Gamma(t_{\ell-1},t_\ell)^C\right) = \prod_{\ell=1}^m \P(\Gamma(t_{\ell-1},t_\ell)^C) \le (1-p)^m \to 0\]
as $m\to\infty$. This implies $\P(\Gamma(t,t_{m}))\to 1$, completing the proof that (1) $\Rightarrow$ (2). 

Now  assume   in addition that the difference~$t-t_0$ in~(1) is bounded uniformly in $t$, say by $\Delta t$. For any $\tau < \sigma$ choose $m$ to be the maximal integer satisfying $\sigma-\tau \ge m \Delta t$. Since $(m+1)\Delta t > \sigma-\tau$ the above argument yields
\[
0 \le 1- \P(\Gamma(\sigma ,\tau)) \le (1-p)^m < C e^{- \gamma (\sigma-\tau)}
\]
with $C:=(1-p)^{-1}$ and $\gamma:=-(\ln(1-p))/\Delta t > 0$. From this estimate we conclude both, the exponential convergence of $\P(\Gamma(t ,t_0))$ as $t_0 \to - \infty$, and that the uniformity condition~\eqref{eq:Tp} holds, choose e.g. $T_q = \gamma^{-1}\ln(C/(1-q))$ for $0<q<1$.
\end{proof}

\section{Random attraction in the nonautonomous TASEP model}\label{sec:TASEPattractors}
We now apply the results of the previous section to the nonautonomous TASEP. To this end, first note that the pullback statements in Theorem~\ref{thm:rattr} and all equivalence statements in Theorem~\ref{thm:singleton} apply to the TASEP NRDS, since the state space is finite and the paths of the 
non-homogeneous
Poisson process are independent on non-overlapping intervals, implying the same for the solutions $\phi$. Hence, the sets~$A(t,\omega)$ from \eqref{eq:Adef} define a global pullback random attractor.

What remains to be  investigated is whether 
\begin{itemize}
\item[(i)] the sets $A(t,\omega)$ from \eqref{eq:Adef} also define a global forward attractor. To this end, we give sufficient conditions on the nonautonomous rates of TASEP such that \eqref{eq:Tp} holds;
\item[(ii)] the sets $A(t,\omega)$ from \eqref{eq:Adef} are singletons; to this end we will give sufficient conditions under which we can verify property~(1) in Theorem~\ref{thm:singleton}, possibly with uniformly bounded $t-t_0$.
\end{itemize}

For addressing both points we rely on the jump order sequences  
$(k_i)_i$ that were introduced in Section \ref{sec:nonhompoisson}. The following assumption on the uniformity of the rates is crucial for this.

\begin{assumption}
    For $R>r>0$ and $m\in\N$, we say that the $n+1$ Poisson processes governing TASEP have  joint $(r,R,m)$-bounded rates on an interval $I=[\tau_1,\tau_2)$, if there exists a division of $I$ into $m$ subintervals $I_i=[q_{i-1},q_{i})\subset I$, $\tau_1=q_0<q_1<\ldots<q_m = \tau_2$, such that 
    \[r \le \int_{q_i}^{q_{i+1}} \lambda_k(t) dt \le R \]
    holds for all $i=0,\ldots,m-1$ and for all $k=0,\ldots,n$. 
\label{asm:boundedrates}    
\end{assumption}
\begin{lemma}
    Consider a bounded interval $I=[\tau_1,\tau_2)$ on which the Poisson processes governing TASEP have joint $(r,R,m)$-bounded rates. Then there exists a value $p(r,R,m,n)>0$ such that for any prescribed tuple of sites $(k_1, \ldots, k_m)$ the probability that this jump sequence occurs in the interval $I=[\tau_1,\tau_2)$ is larger than $p(r,R,m,n)$. 
\label{lemma:jumpsequence}    
\end{lemma}
\begin{proof}
For any of the $n+1$ Poisson processes governing TASEP, the probability that this Poisson process has exactly one event in the interval $I_i$ is greater or equal 
$p_1(r,R):=\min\{re^{-r},\,Re^{-R}\}$. Similarly, for any process the probability that it has no event on $I_i$ is greater or equal $p_2(R)=e^{-R}$. As the $n+1$ Poisson processes are independent, the probability that in the interval $I_i$ exactly one jump event occurs and that this event is caused by process $k_i$ is larger 
than~$p_1 (r,R)p_2(R)^n$. As the jump events on different subintervals $I_i\ne I_j$ are independent, the event that this happens for all $I_i$ is thus greater or equal to~$p(r,R,m,n):=(p_1(r,R) p_2(R)^n)^m$.
\end{proof}

Based on this lemma, we can state the following result for the set~$\Gamma(t,t_0)$ from~\eqref{eq:gammadef}.

\begin{lemma} 
Suppose that the non-autonomous TASEP satisfies Assumption \ref{asm:boundedrates} for an interval $I=[t_0,t)$ with $m=n(n+1)/2$. Then there exists $\delta=\delta(r,R,n)>0$ such that \[ \P(\Gamma(t_0,t)) \ge \delta. \]
\label{lemma:fintime}
\end{lemma}
\begin{proof}
According to Lemma \ref{lemma:jumpsequence}, the probability that on the interval $I$ the jump order sequence 
\[  n, \;\; n-1, n, \;\; n-2, n-1, n, \;\; n-3, \ldots,n, \;\; \ldots,
 \;\;1,2, \ldots,n \]
occurs is greater or equal to $\delta(r,R,n)=p(r,R,n(n+1)/2,n)>0$.
Now, for all $\omega\in\Omega$ generating this sequence it is easily seen that $\phi(t,t_0,x,\omega) = (0,\ldots,0)$ for all $x\in X$. This shows the claim.
\end{proof}

\begin{remark} 
The probability bound $\delta>0$ that follows from the proof of Lemma \ref{lemma:fintime} is clearly not optimal. Moreover, the condition in Assumption \ref{asm:boundedrates} is only sufficient but not necessary. Yet, as Example \ref{weirdTASEP} and Remark \ref{rem:weirdTASEP}, below, show, certain bounds on the rates are needed for obtaining forward and pullback attraction, hence we cannot do without any conditions on the rates.
\end{remark}

Using Lemma \ref{lemma:fintime}, we can now state the following theorem. 

\begin{theorem} Consider the non-autonomous TASEP model.

(a) Assume that there exist $R>r>0$ such that for each $t\in\R$ there is $t_0<t$ such that Assumption \ref{asm:boundedrates} holds for  this pair $(R, r)$ on the interval $I=[t_0,t)$ with $m=n(n+1)/2$.
Then the sets $A(t,\omega)$ from \eqref{eq:Adef} define a global random pullback attractor, where $A(t,\omega)$ is a singleton for all $t\in\R$ and almost all $\omega\in\Omega$. 

(b) If, moreover, there exist $R>r>0$ and $\Delta t>0$ such that Assumption \ref{asm:boundedrates} holds for  this pair $(R, r)$ with $m=n(n+1)/2$ on all intervals of the form $I=[\tau,\tau+\Delta t)$, $\tau\in\R$,
then the $A(t,\omega)$ also provide a global random forward attractor and the convergence in Theorem \ref{thm:singleton}(2) has an exponential rate.
\label{thm:TASEPmain}
\end{theorem}
\begin{proof} (a) Theorem \ref{thm:rattr} yields that $A(t,\omega)$ is a pullback attractor. The sets $A(t,\omega)$ are a.s.\ singletons because Lemma \ref{lemma:fintime} implies property (1) in Theorem \ref{thm:singleton} and since the Poisson processes in TASEP are independent on non-overlapping intervals, this implies Theorem \ref{thm:singleton}(4).

(b) Under the additional assumption on the uniformity of $\Delta t$ the claim follows immediately from Lemma~\ref{lemma:fintime} and from the last statement of Theorem~\ref{thm:singleton}.
\end{proof}

\begin{remark}\label{rem:nonsync}
Statement (b) in Theorem \ref{thm:TASEPmain} in particular implies that any two solutions of TASEP synchronize almost surely to a single trajectory after sufficiently large time. As in the time-invariant case, this almost sure identity does not exclude the existence of non-trivial jump time sequences for which $\phi(t,t_0,x_1,\omega)$ and $\phi(t,t_0,x_2,\omega)$ never coincide. We refer to \cite{GrKM21} for an example, which also shows that in TASEP forward attraction of~$A$ does not hold for every~$\omega\in\Omega$. 
\end{remark}

The following example shows that forward attraction may indeed fail to hold if the assumption in Theorem~\ref{thm:TASEPmain}~(b) is violated.

\begin{example}\label{weirdTASEP}
We define jump rates for the non-autonomous TASEP model with~$n=2$ sites
that satisfy the assumptions of Theorem~\ref{thm:TASEPmain}~(a), but do \emph{not} satisfy the additional uniformity assumption of Theorem~\ref{thm:TASEPmain}~(b), and  prove that the sets~$A(t, \omega)$ do not provide a global random forward attractor. Our example is based on the observation that the jump order sequence $1,2,1,0$ takes the state $x_1:=10$ to $x_1$,  and the state $x_2:=11$ to $x_2$. We define the rates $\lambda_i(t)=1$ for all $t<0$ and $i=0,1,2$. For $t\ge 0$ the rates are $0$, except for
\begin{eqnarray*}
\lambda_0(t)&=&4(j+1)\,, \quad \mbox{ if } t\in [j+3/4, j+1)\,,\\
\lambda_1(t)&=&4(j+1)\,, \quad \mbox{ if } t\in [j,j+1/4) \cup [j+1/2, j+3/4)\,,\\
\lambda_2(t)&=&4(j+1)\,, \quad \mbox{ if } t\in [j+1/4, j+1/2)
\end{eqnarray*}
for all non-negative integers $j$. Note that for all $t\ge 0$ exactly one of the rates $\lambda_i(t)$ is positive which simplifies the analysis. E.g.~in the interval $[0,1/4)$ only jumps starting at site $1$ may occur and they occur with probability $1-e^{-1}$ if the system is in state $x_1=10$ at time $t=0$. Observe that at time $t=1/4$ the system is then in state $01$ no matter how many additional times the jump from site~$1$ is attempted in the interval~$[0,1/4)$. For non-negative integers~$j$, let 
\[
\Omega_j:=\{\omega \in \Omega \mid \mbox{at least one jump attempt occurs in each [j+k/4, j+(k+1)/4) }, k=0,1,2,3 \}.
\]
Using independence one obtains $\P(\Omega_j)=(1-e^{-(j+1)})^4$. Moreover, for all $\omega \in \Omega_j$ we have $\phi(j+1,j,\{x_1\},\omega) = \{x_1\}$ and $\phi(j+1,j,\{x_2\},\omega) = \{x_2\}$. Set $\tilde \Omega := \bigcap_{j\ge 0} \Omega_j$. Then independence implies $\P(\tilde \Omega)=\Pi_{j\ge 0}(1-e^{-(j+1)})^4 > 0$. In addition, for all~$\omega \in \tilde \Omega$ and~$j \in \N$ we have $\{x_1, x_2\} \subset \phi(j,0,X,\omega)$. As our jump rates satisfy the assumptions of Theorem~\ref{thm:TASEPmain}~(a) we know that the global random pullback attractor $A(t,\omega)$ is a singleton for all $t \in \R$ and almost all $\omega \in \Omega$. Since $\phi(t,0,X,\omega)$ contains at least two elements for all $t\ge 0$ and $\omega \in \tilde \Omega$ we learn that $\lim_{t\to \infty} \dist\big(\phi(t, 0, X, \omega), A(t,\omega)\big) = 1$ with positive probability. Therefore the sets $A(t, \omega)$ do not provide a global random forward attractor.
\end{example}

\begin{remark}\label{rem:weirdTASEP}
(i)
The construction of Example~\ref{weirdTASEP} can easily be extended to lattices of arbitrary size $n$. Replace state~$x_1$ by the state where only the second site is empty,  and~$x_2$ by the state where no site is empty. These two states are preserved by the jump order sequence $1, n, n-1, \ldots, 2, 1, 0$. 

(ii) One may modify Example~\ref{weirdTASEP} so that it does not satisfy the assumptions of Theorem~\ref{thm:TASEPmain}~(a). Set $\lambda_i(t) = 1$ for $t \geq 0$ and $i=0,1,2$. For $t<0$ we use the definition in
Example~\ref{weirdTASEP} for negative integers $j$ and with rates $-4j$ instead of $4(j+1)$ in the corresponding intervals. Then one can show for $t \le 0$ that both states $x_1$ and $x_2$ from Example~\ref{weirdTASEP} are contained in the pullback attractor $A(t, \omega)$ with positive probability. Nevertheless, as $t\to \infty$ the sets $A(t, \omega)$ become singletons almost surely and provide a global random forward attractor.
\end{remark}
The following corollary specializes the main result in Theorem \ref{thm:TASEPmain} to the case of periodic rates.

\begin{corollary}\label{cor:periodic}
    Consider the non-autonomous TASEP model in which all rates $\lambda_0,\ldots,\lambda_n:\R\to\R_{\ge 0}$ are periodic with identical (not necessarily minimal) period $T_{per}>0$, i.e., $\lambda_i(t+T_{per}) = \lambda_i(t)$ holds for all $t\in\R$ and all $i=0,\ldots,n$. Assume that $\lambda_i|_{[0,T_{per}]}$, $i=0,\ldots,n$ are $L^1$ functions that are not equal to the zero-function. Then the statements~(a) and~(b) from Theorem \ref{thm:TASEPmain} hold with $\Delta t = \frac{n(n+1)}{2}T_{per}$. 
\end{corollary}
\begin{proof}
    Since the $\lambda_i$ restricted to $[0,T_{per}]$ are $L^1$-functions that are not identically $0$, it follows that there are $R>r>0$ such that 
    \[ r \le \int_0^{T_{per}} \lambda_i(t) dt \le R. \]
    Now periodicity implies that 
    \[\int_t^{t+T_{per}} \lambda_i(t) dt = \int_0^{T_{per}} \lambda_i(t) dt\]
    holds for all $t\in\R$. This implies that the assumptions of Theorem \ref{thm:TASEPmain}(a) and (b) are satisfied and thus also the statements hold.
\end{proof}

\section{Numerical simulations of synchronization}\label{sec:NumSim}

In this section, we show two numerical simulations\footnote{Videos showing the entire dynamics of the simulations are available at 
\href{https://mms.uni-bayreuth.de/Panopto/Pages/Sessions/List.aspx?folderID=6e4e74be-5724-4b89-83f8-b26f0074f233}{https://mms.uni-bayreuth.de/Panopto/Pages/Sessions/List.aspx?folderID=6e4e74be-5724-4b89-83f8-b26f0074f233}.}. We assume that all the rates have units of~$1/\text{sec}$.

In the first simulation, we consider TASEP with $n=20$ sites, in which all rates except for $\lambda_{11}(t)$ are constantly equal to one, while $\lambda_{11}(t)$ changes periodically every 20 time units from 1 to $10^{-6}$ and vice versa. A green or red ``traffic light'' in the figures indicates the current value of $\lambda_{11}(t)$, where red corresponds to $10^{-6}$ and green to~$1$.

We have run the simulations with three different seeds of the random generator. The states at times $t=0,10,20,\ldots,100$ are depicted in the three columns of Figure \ref{fig:trafficlight}. It can be seen  that in all three runs synchronization happens between $t=30\text{sec}$ and $t=40\text{sec}$.

\begin{figure}
\begin{minipage}{1.4cm}$t=0$\\[6mm]
\end{minipage}
\includegraphics[trim=24 110 24 120, clip, width=4cm]{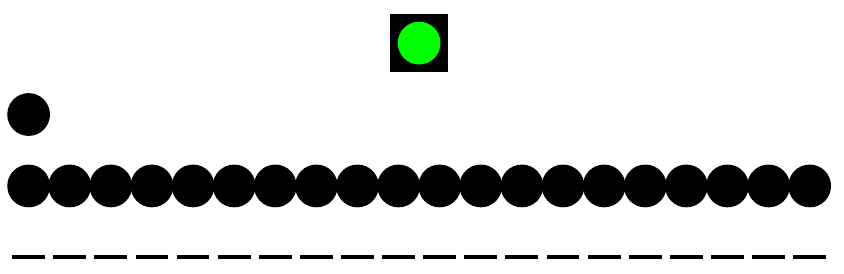}\qquad
\includegraphics[trim=24 110 24 120, clip, width=4cm]{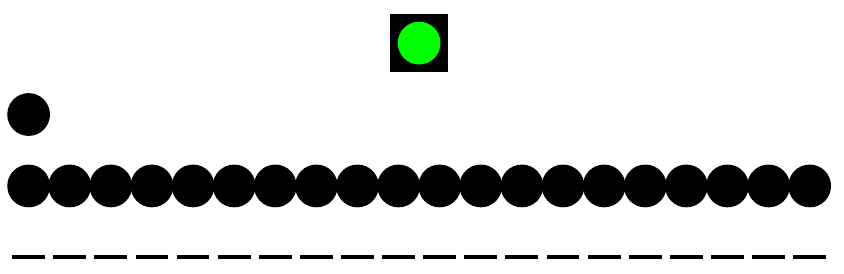}\qquad
\includegraphics[trim=24 110 24 120, clip, width=4cm]{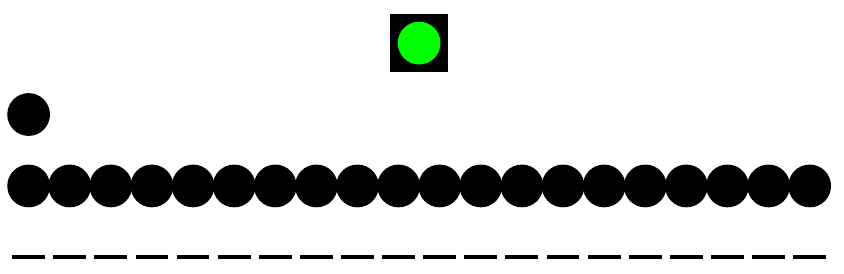}

\begin{minipage}{1.4cm}$t=10$\\[6mm]
\end{minipage}
\includegraphics[trim=24 110 24 120, clip, width=4cm]{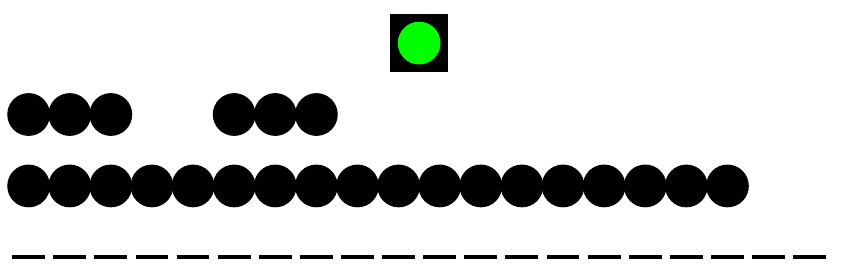}\qquad
\includegraphics[trim=24 110 24 120, clip, width=4cm]{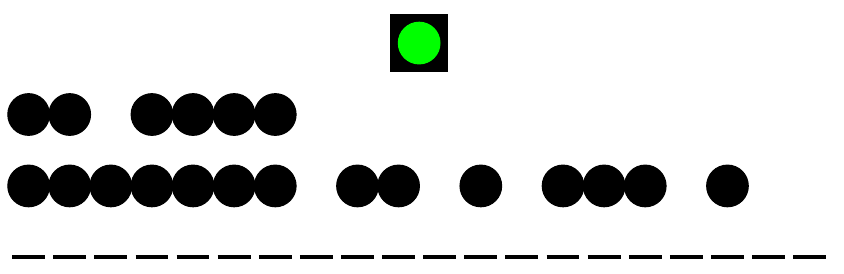}\qquad
\includegraphics[trim=24 110 24 120, clip, width=4cm]{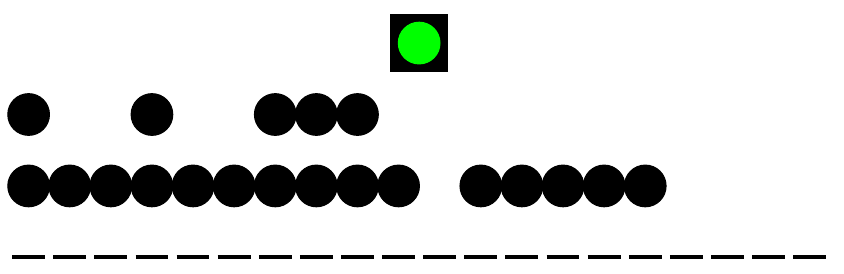}

\begin{minipage}{1.4cm}$t=20$\\[6mm]
\end{minipage}
\includegraphics[trim=24 110 24 120, clip, width=4cm]{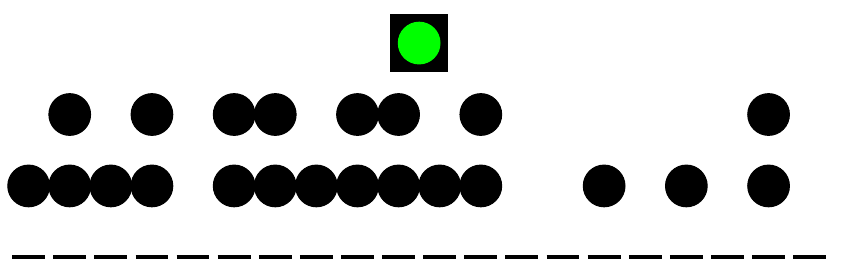}\qquad
\includegraphics[trim=24 110 24 120, clip, width=4cm]{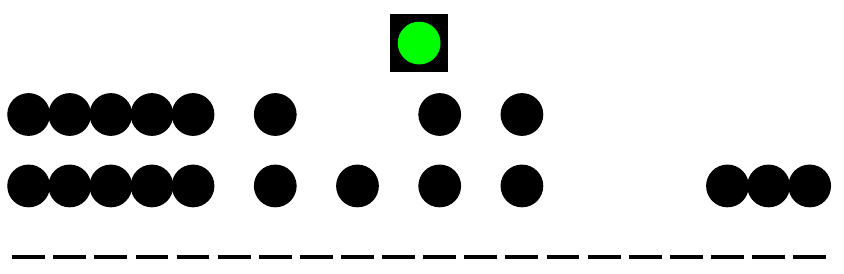}\qquad
\includegraphics[trim=24 110 24 120, clip, width=4cm]{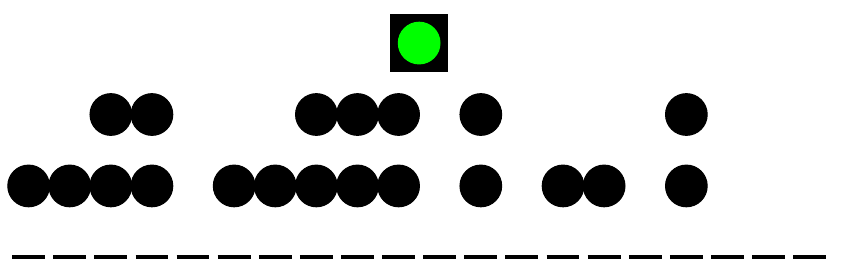}

\begin{minipage}{1.4cm}$t=30$\\[6mm]
\end{minipage}
\includegraphics[trim=24 110 24 120, clip, width=4cm]{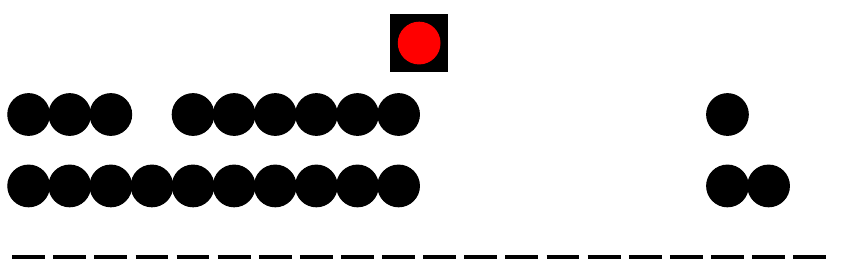}\qquad
\includegraphics[trim=24 110 24 120, clip, width=4cm]{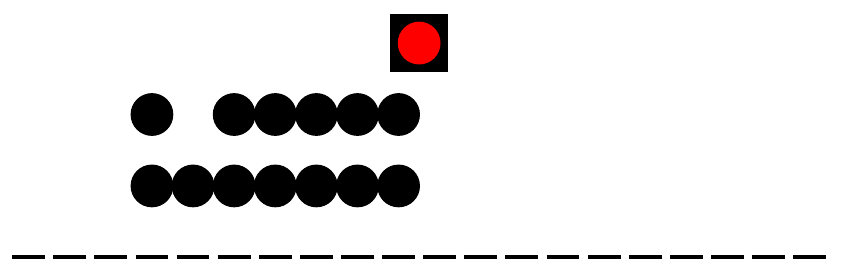}\qquad
\includegraphics[trim=24 110 24 120, clip, width=4cm]{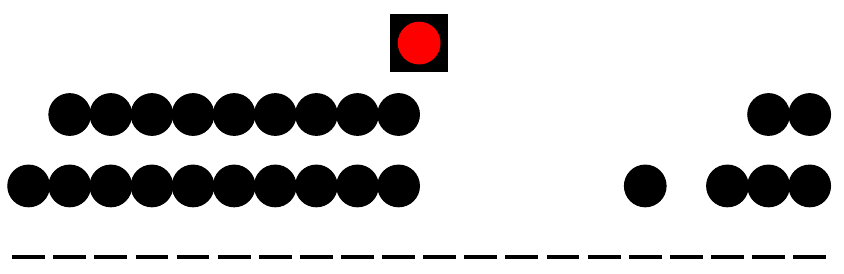}

\begin{minipage}{1.4cm}$t=40$\\[6mm]
\end{minipage}
\includegraphics[trim=24 110 24 120, clip, width=4cm]{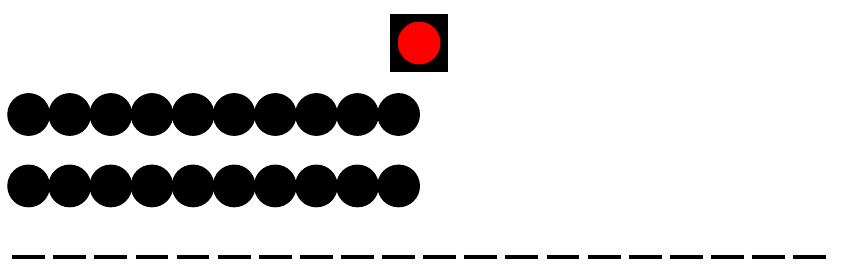}\qquad
\includegraphics[trim=24 110 24 120, clip, width=4cm]{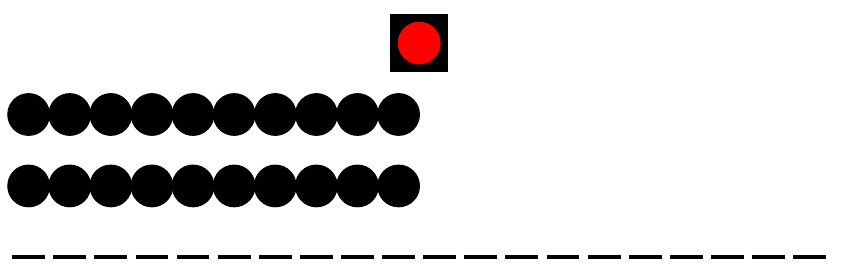}\qquad
\includegraphics[trim=24 110 24 120, clip, width=4cm]{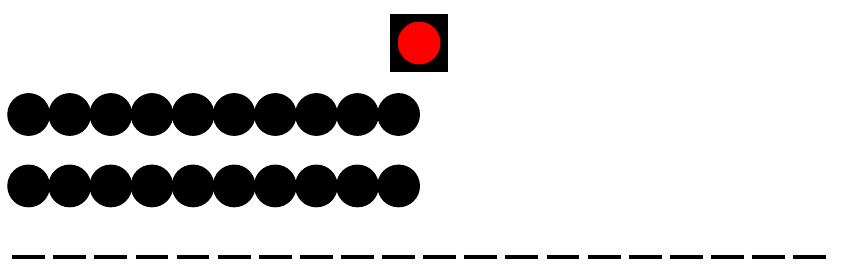}

\begin{minipage}{1.4cm}$t=50$\\[6mm]
\end{minipage}
\includegraphics[trim=24 110 24 120, clip, width=4cm]{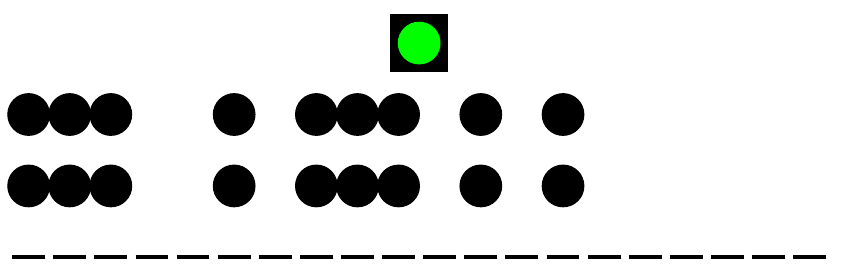}\qquad
\includegraphics[trim=24 110 24 120, clip, width=4cm]{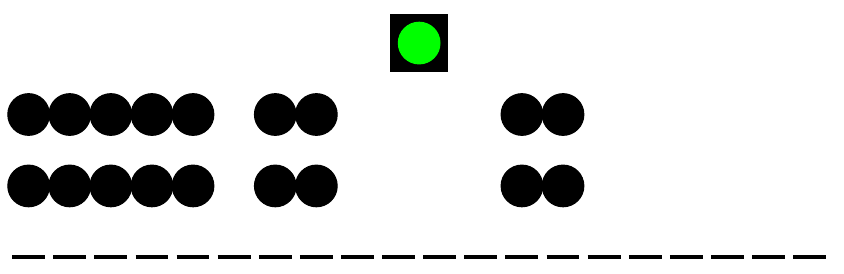}\qquad
\includegraphics[trim=24 110 24 120, clip, width=4cm]{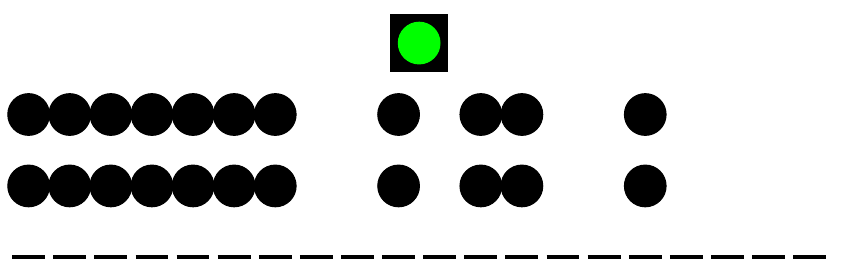}

\begin{minipage}{1.4cm}$t=60$\\[6mm]
\end{minipage}
\includegraphics[trim=24 110 24 120, clip, width=4cm]{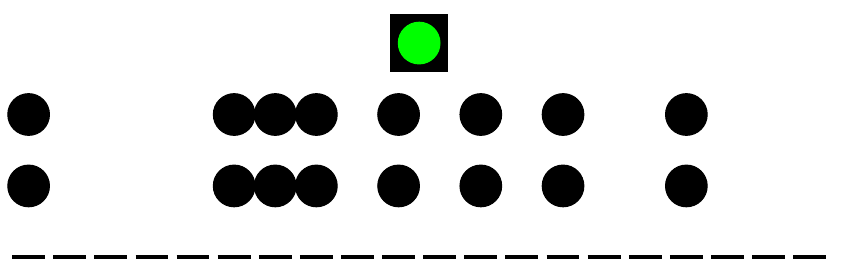}\qquad
\includegraphics[trim=24 110 24 120, clip, width=4cm]{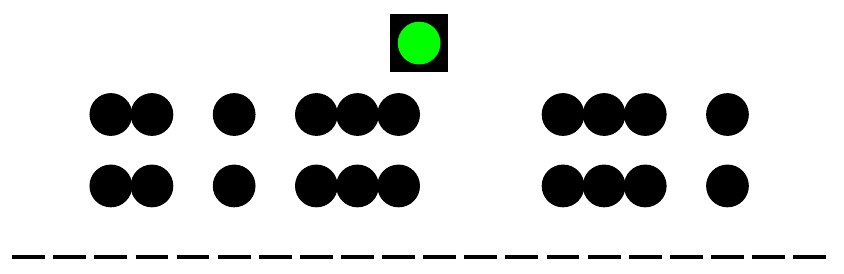}\qquad
\includegraphics[trim=24 110 24 120, clip, width=4cm]{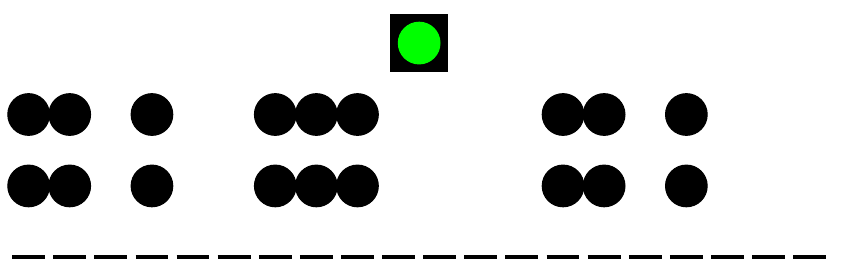}

\begin{minipage}{1.4cm}$t=70$\\[6mm]
\end{minipage}
\includegraphics[trim=24 110 24 120, clip, width=4cm]{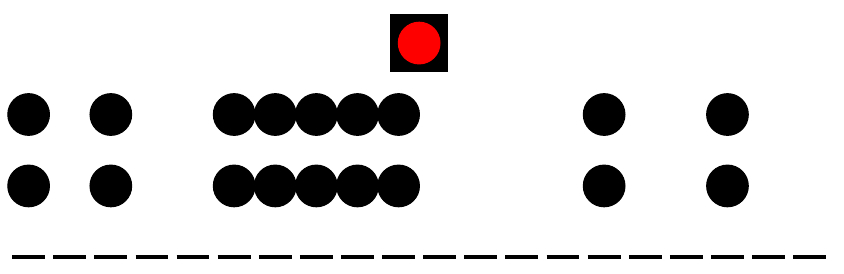}\qquad
\includegraphics[trim=24 110 24 120, clip, width=4cm]{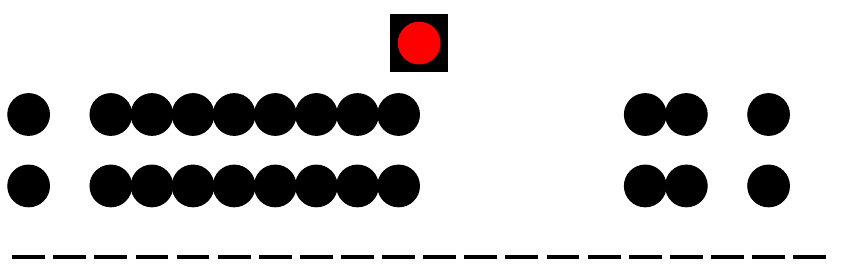}\qquad
\includegraphics[trim=24 110 24 120, clip, width=4cm]{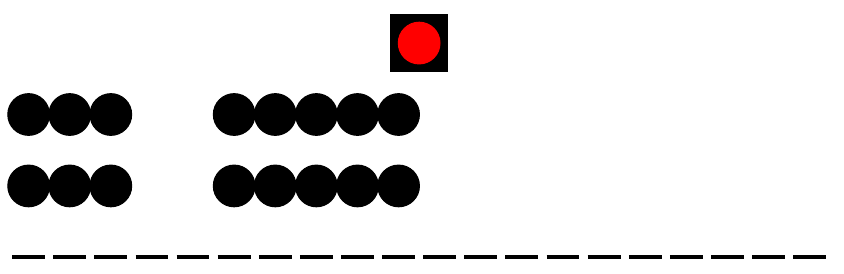}

\begin{minipage}{1.4cm}$t=80$\\[6mm]
\end{minipage}
\includegraphics[trim=24 110 24 120, clip, width=4cm]{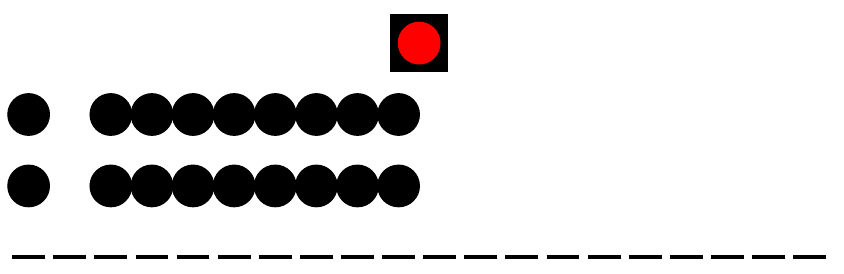}\qquad
\includegraphics[trim=24 110 24 120, clip, width=4cm]{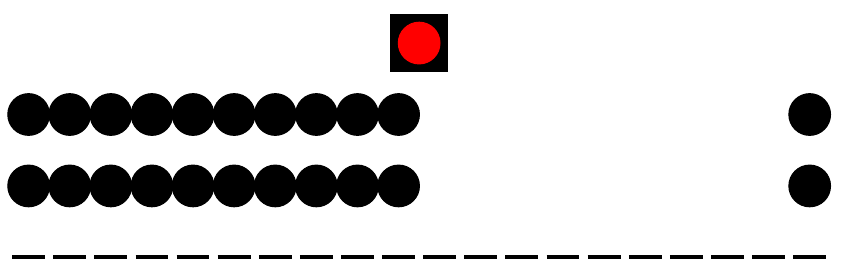}\qquad
\includegraphics[trim=24 110 24 120, clip, width=4cm]{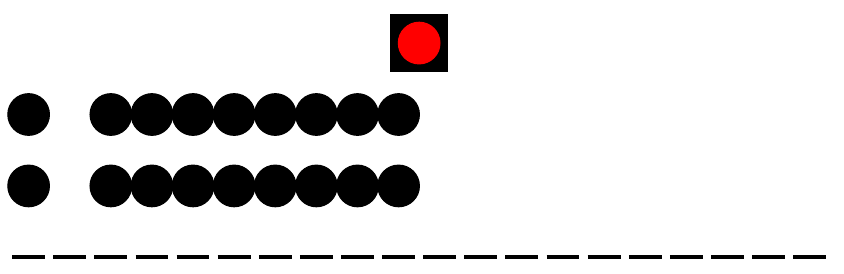}

\begin{minipage}{1.4cm}$t=90$\\[6mm]
\end{minipage}
\includegraphics[trim=24 110 24 120, clip, width=4cm]{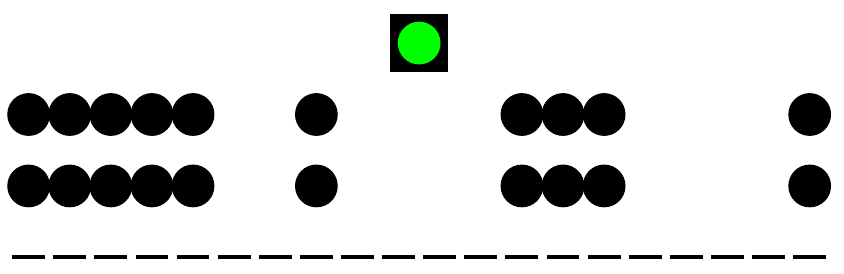}\qquad
\includegraphics[trim=24 110 24 120, clip, width=4cm]{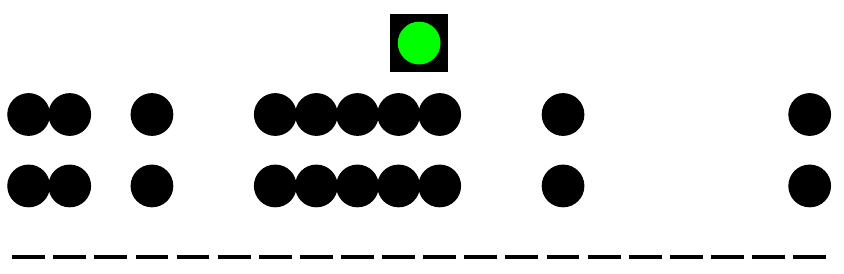}\qquad
\includegraphics[trim=24 110 24 120, clip, width=4cm]{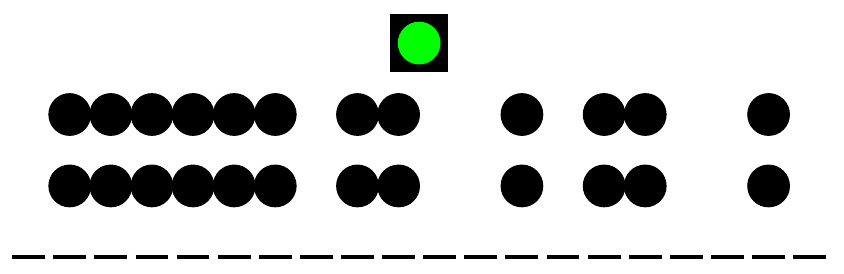}

\begin{minipage}{1.4cm}$t=100$\\[6mm]
\end{minipage}
\includegraphics[trim=24 110 24 120, clip, width=4cm]{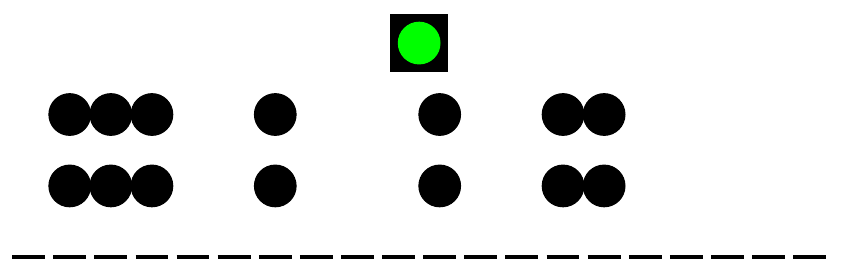}\qquad
\includegraphics[trim=24 110 24 120, clip, width=4cm]{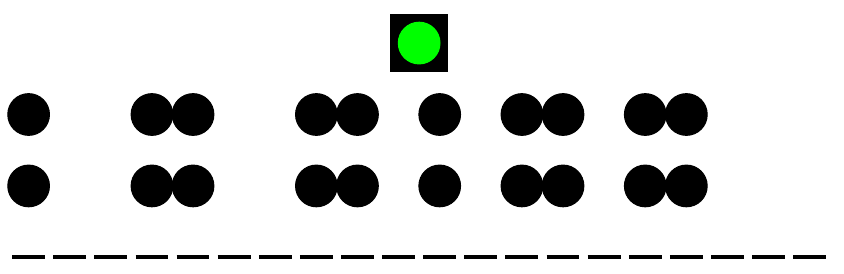}\qquad
\includegraphics[trim=24 110 24 120, clip, width=4cm]{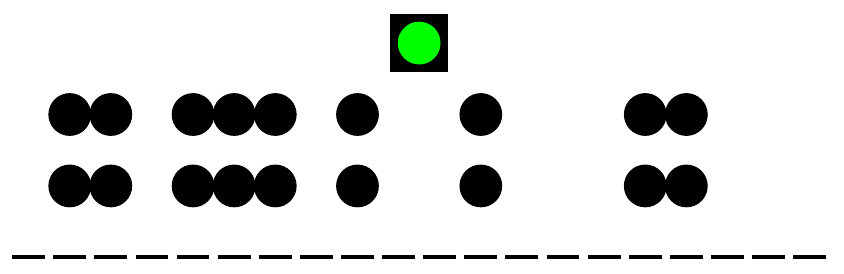}
\caption{Six trajectories of TASEP with $n=20$ sites,   varying rate $\lambda_{11}(t)$ for three different random jump time sequences (left to right) and two different initial conditions (depicted on top of each other for each $t$). The red light indicates that $\lambda_{11}(t)=10^{-6}$, while the green light indicates that $\lambda_{11}(t)=1$. Synchronization occurs for each run 
between~$t=30\text{sec}$ and~$t=40\text{sec}$.}
\label{fig:trafficlight}\end{figure}

The second example is a simulation of the nonautonomous TASEP with~$n=2$ sites in 
Example \ref{weirdTASEP} above, where, however, to increase the probability of non-synchronization in forward time, we have replaced the values $4(j+1)$ for the $\lambda_i(t)$s by $20(j+1)$. Indeed, for all simulations we performed the two trajectories did \emph{not} synchronize. Figure~\ref{fig:nonattraction} shows figures from one such run of the model.

\begin{figure}
\begin{center}
\begin{minipage}{1.4cm}$t=0.00$\\[6mm]
\end{minipage}
\includegraphics[trim=24 110 24 120, clip, width=4cm]{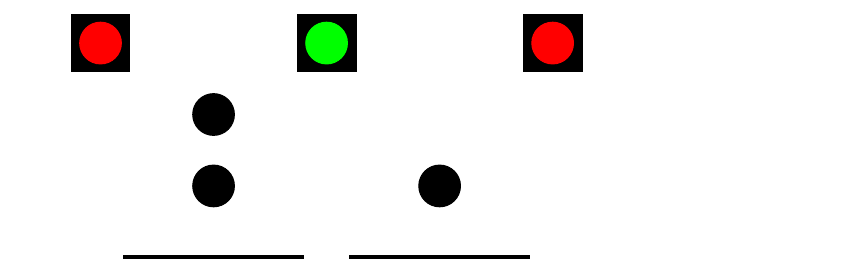}
\begin{minipage}{1.4cm}$t=2.00$\\[6mm]
\end{minipage}
\includegraphics[trim=24 110 24 120, clip, width=4cm]{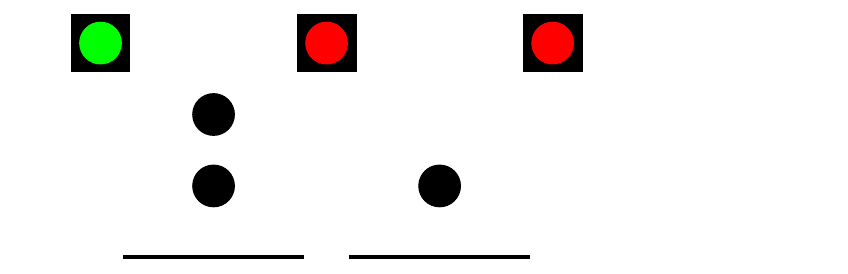}

\begin{minipage}{1.4cm}$t=0.25$\\[6mm]
\end{minipage}
\includegraphics[trim=24 110 24 120, clip, width=4cm]{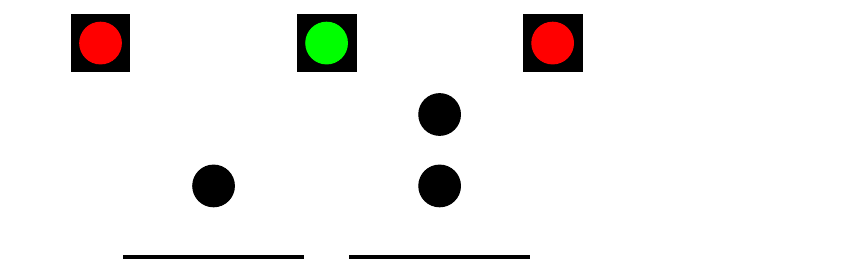}
\begin{minipage}{1.4cm}$t=2.25$\\[6mm]
\end{minipage}
\includegraphics[trim=24 110 24 120, clip, width=4cm]{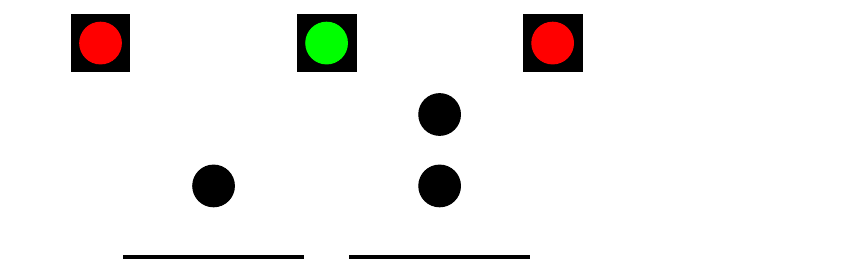}

\begin{minipage}{1.4cm}$t=0.50$\\[6mm]
\end{minipage}
\includegraphics[trim=24 110 24 120, clip, width=4cm]{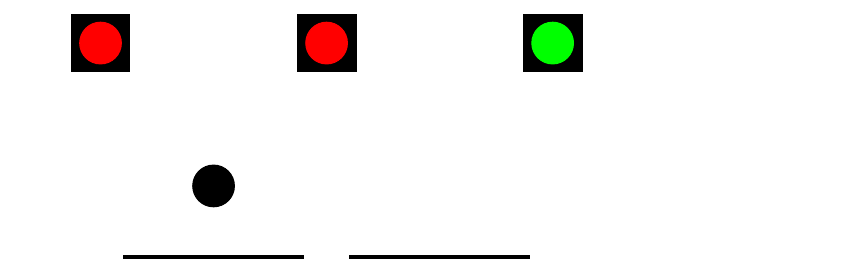}
\begin{minipage}{1.4cm}$t=2.50$\\[6mm]
\end{minipage}
\includegraphics[trim=24 110 24 120, clip, width=4cm]{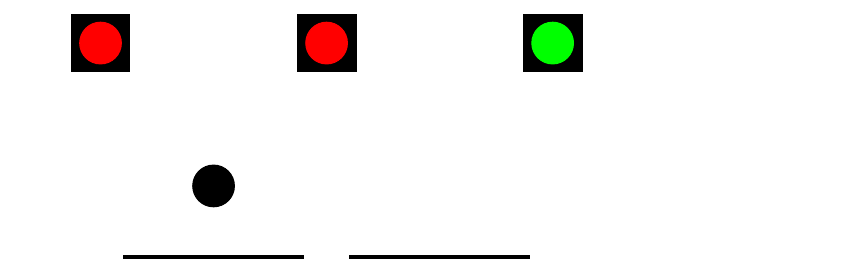}

\begin{minipage}{1.4cm}$t=0.75$\\[6mm]
\end{minipage}
\includegraphics[trim=24 110 24 120, clip, width=4cm]{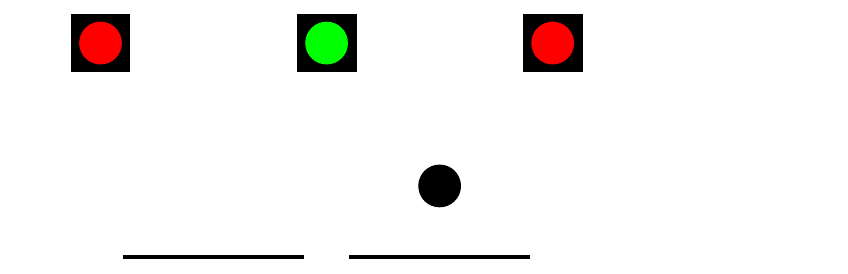}
\begin{minipage}{1.4cm}$t=2.75$\\[6mm]
\end{minipage}
\includegraphics[trim=24 110 24 120, clip, width=4cm]{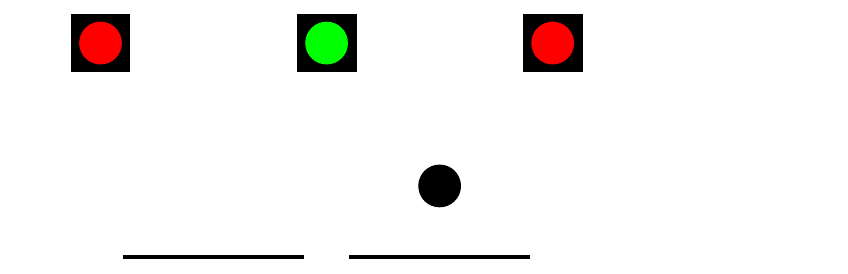}

\begin{minipage}{1.4cm}$t=1.00$\\[6mm]
\end{minipage}
\includegraphics[trim=24 110 24 120, clip, width=4cm]{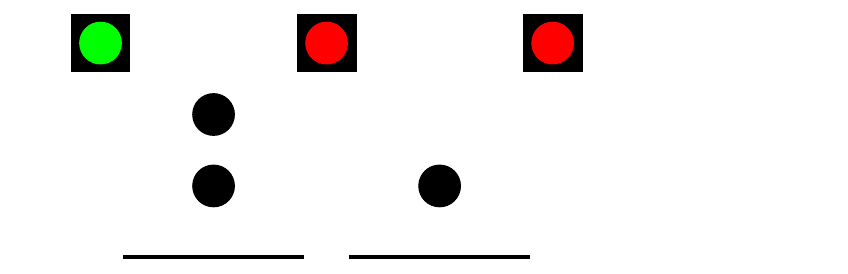}
\begin{minipage}{1.4cm}$t=3.00$\\[6mm]
\end{minipage}
\includegraphics[trim=24 110 24 120, clip, width=4cm]{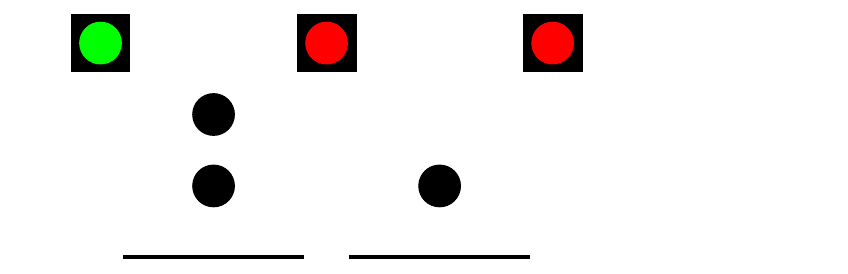}

\begin{minipage}{1.4cm}$t=1.25$\\[6mm]
\end{minipage}
\includegraphics[trim=24 110 24 120, clip, width=4cm]{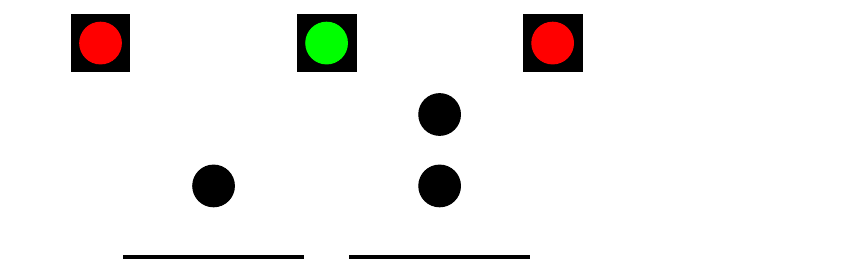}
\begin{minipage}{1.4cm}$t=3.25$\\[6mm]
\end{minipage}
\includegraphics[trim=24 110 24 120, clip, width=4cm]{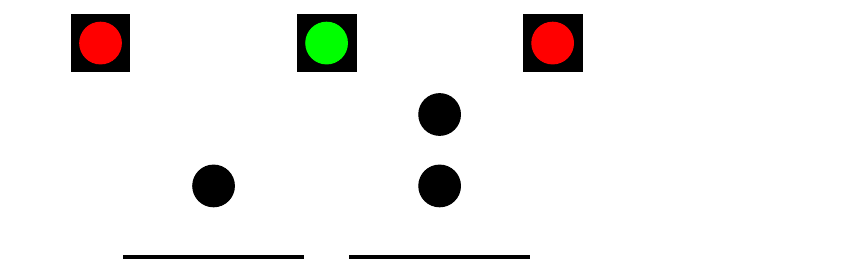}

\begin{minipage}{1.4cm}$t=1.50$\\[6mm]
\end{minipage}
\includegraphics[trim=24 110 24 120, clip, width=4cm]{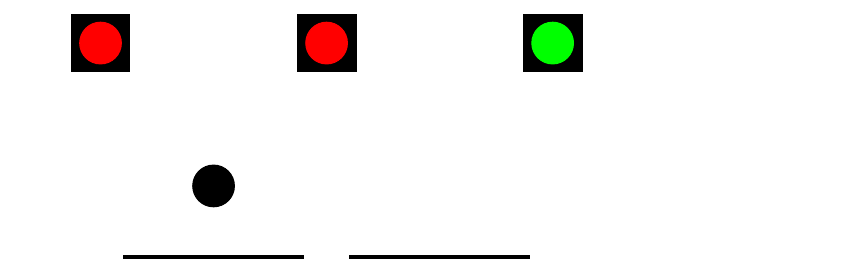}
\begin{minipage}{1.4cm}$t=3.50$\\[6mm]
\end{minipage}
\includegraphics[trim=24 110 24 120, clip, width=4cm]{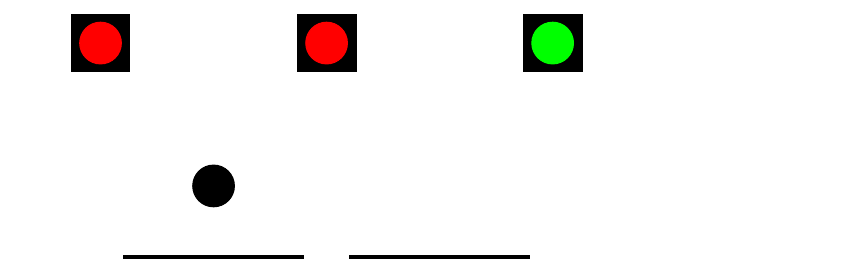}

\begin{minipage}{1.4cm}$t=1.75$\\[6mm]
\end{minipage}
\includegraphics[trim=24 110 24 120, clip, width=4cm]{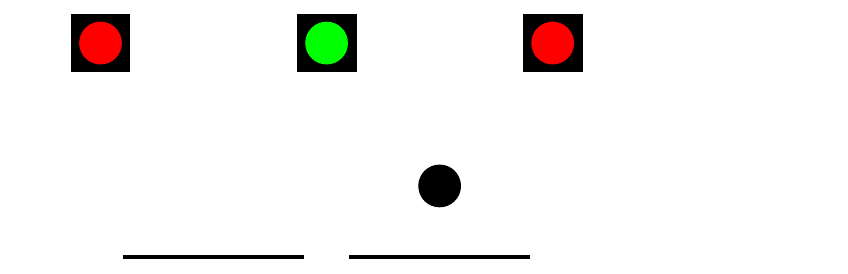}
\begin{minipage}{1.4cm}$t=3.75$\\[6mm]
\end{minipage}
\includegraphics[trim=24 110 24 120, clip, width=4cm]{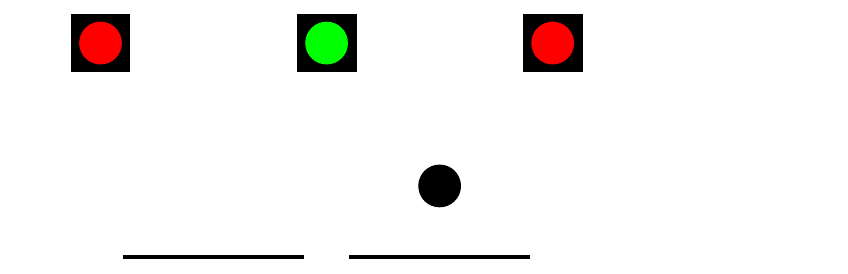}
\end{center}
\caption{Two trajectories of TASEP with $n=2$ sites and with rates varying as in Example~\ref{weirdTASEP} with modifications as described in Section \ref{sec:NumSim}. No synchronization occurs in forward time.}
\label{fig:nonattraction}\end{figure}

\newpage

\section{Conclusion} 

 TASEP 
is a fundamental model for the unidirectional movement of particles along a 1D  track.   It is typically assumed that the stochastic  hopping rates are time-invariant, but there are good reasons to study TASEP 
with time-varying hopping rates. For example, in biological applications the rates may depend on periodic processes like the cell cycle program or the
circadian rhythm.

 It is known that if the hopping rates are jointly~$T$-periodic then the probability distribution of the solutions converges to a unique $T$-periodic distribution, which is independent of the initial condition. Here, we considered the more detailed 
behavior of the individual paths of the model. Our main results are based on 
formulating  TASEP with time-varying hopping rates as an~NRDS, and 
 providing  conditions under which nonautonomous random pullback and forward
attractors exist for a general finite-state~NRDS. Using this approach, we provide  rather tight conditions guaranteeing that the attractors are singletons implying almost sure synchronization of the paths.

\bibliographystyle{IEEEtranS}

\bibliography{TASEP_attractor} 

\end{document}